\documentclass[12pt,reqno]{amsart}
\usepackage{amssymb}
\usepackage{mathtools}
\usepackage{mathabx}
\usepackage{mathrsfs}
\usepackage{verbatim}  
\usepackage{wrapfig}
\usepackage{xcolor}
\usepackage{hyperref}
\usepackage{enumerate}

\hypersetup{
	colorlinks=true,
	linkcolor=blue,
	filecolor=magenta,      
	urlcolor=cyan,
	citecolor={green!70!black},
}

\newtheorem{thm}{Theorem}[section]
\newtheorem{lem}[thm]{Lemma}

\newtheorem{cor}[thm]{Corollary}

\newtheorem{Def}[thm]{Definition}

\newtheorem{prop}[thm]{Proposition}
\newtheorem{rem}[thm]{Remark}
\newtheorem{ex}[thm]{Example}

\newcommand{\bdfn}{\begin{Def} \rm}
	\newcommand{\edfn}{\end{Def}}

\newtheorem{prob}[thm]{Problem}

\newcommand{\trcp}{\textrm{r.c.p.}}
\newcommand{\ercp}{\emph{r.c.p.}}

\newcommand{\es}{\emptyset}
\newcommand{\ci}{\subseteq}

\newcommand{\al}{\alpha}

\newcommand{\mb}{\mathbb}
\newcommand{\mc}{\mathcal}

\newcommand{\msc}{\mathscr}

\newcommand{\beqa}{\begin{eqnarray*}}
	\newcommand{\eeqa}{\end{eqnarray*}}

\newcounter{cnt1}
\newcounter{cnt2}
\newcounter{cnt3}
\newcounter{cnt4}
\newcommand{\blr}{\begin{list}{$($\roman{cnt1}$)$} {\usecounter{cnt1}
			\setlength{\topsep}{0pt} \setlength{\itemsep}{0pt}}}
	\newcommand{\blR}{\begin{list}{\Roman{cnt4}.\ } {\usecounter{cnt4}
				\setlength{\topsep}{0pt} \setlength{\itemsep}{0pt}}}
		\newcommand{\bla}{\begin{list}{$(\alph{cnt2})$} {\usecounter{cnt2}
					\setlength{\topsep}{0pt} \setlength{\itemsep}{0pt}}}
			\newcommand{\bln}{\begin{list}{$($\arabic{cnt3}$)$} {\usecounter{cnt3}
						\setlength{\topsep}{0pt} \setlength{\itemsep}{0pt}}}
				\newcommand{\el}{\end{list}}

			\title[Restricted Chebyshev centers in $L_1$-predual spaces]{Restricted Chebyshev centers in $L_1$-predual spaces}
			\author[Thomas]{Teena Thomas}
			\address[Teena Thomas]{Department of Mathematics \\
				Indian Institute of Technology Hyderabad \\
				India, \textit{E-mail~:} \textit{tteena.tthomas@gmail.com/ma19resch11003@iith.ac.in}}
			
\begin{document}
				\maketitle
				
				\begin{abstract}
					In this paper, we provide a necessary and sufficient condition for the existence of a restricted Chebyshev center of a compact subset of an $L_{1}$-predual space in a closed convex subset of the $L_{1}$-predual space. We also provide a geometrical characterization of an $L_{1}$-predual space in terms of the restricted Chebyshev radius in the following manner. A real Banach space $X$ is an $L_{1}$-predual space if and only if for each non-empty finite subset $F$ of $X$ and closed convex subset $V$ of $X$, $\emph{rad}_{V}(F) = \emph{rad}_{X}(F) + d(V, \emph{cent}_{X}(F))$, where we denote $\emph{rad}_{X}(F)$, $\emph{rad}_{V}(F)$, $\emph{cent}_{X}(F)$ and $d(V, \emph{cent}_{X}(F))$ to be the Chebyshev radius of $F$ in $X$, the restricted Chebyshev radius of $F$ in $V$, the set of Chebyshev centers of $F$ in $X$  and the distance between the sets $V$ and $\emph{cent}_{X}(F)$ respectively. Furthermore, we explicitly describe the Chebyshev centers of closed bounded subsets of an $M$-summand in the space of real-valued continuous functions on a compact Hausdorff space.
				\end{abstract}
			\section{Introduction}
			\label{intro}
			In this paper, we consider Banach spaces only over the real field $\mb{R}$ and all the subspaces are assumed to be norm closed. Let $X$ be a Banach space. For $x \in X$ and $r>0$, $B(x,r)$ denotes the norm open ball in $X$ centered at $x$ with radius $r$ and $B_X$ denotes the norm closed unit ball of $X$. Let $\mc{CV}(X)$, $\mc{CB}(X)$, $\mc{K}(X)$ and $\mc{F}(X)$ denote the classes of non-empty closed convex, closed bounded, compact and finite subsets of $X$ respectively. For two subsets $A$ and $B$ of $X$, $d(A,B) = \inf \{\|a-b\|: a\in A\mbox{, }b \in B\}$. 
			
			For an element $x \in X$, a set $B \in \mc{CB}(X)$, let $r(x,B) = \sup\{\|x-b\|:b \in B\}$. Let $V \in \mc{CV}(X)$. The quantity $\emph{rad}_{V}(B) = \inf_{v \in V} r(v,B)$ is called the \textit{restricted Chebyshev radius} of $B$ in $V$. An element $v \in V$ such that $\emph{rad}_{V}(B) = r(v,B)$ is called a \textit{restricted  Chebyshev center} of $B$ in $V$ and the set $\{v \in V: \emph{rad}_{V}(B) = r(v,B)\}$ is denoted by $\emph{cent}_{V}(B)$. If $V = X$, then $\emph{rad}_{X}(B)$ is called the \textit{Chebyshev radius} of $B$ in $X$ and the elements in $\emph{cent}_{X}(B)$ are called the \textit{Chebyshev centers} of $B$ in $X$. Furthermore, for $\delta>0$, let $\emph{cent}_{V}(B,\delta)=\{v \in V: r(v,B) \leq \emph{rad}_{V}(B)+\delta\}$.
			
			\begin{Def}[{\cite{PN}}]
				Let $X$ be a Banach space, $V \in \mc{CV}(X)$ and $\msc{F} \ci \mc{CB}(X)$. 
				\begin{enumerate}[{(i)}]
					\item We say $X$ admits centers for $\msc{F}$ if for each $F \in \msc{F}$, $\textrm{cent}_{X}(F) \neq \es$.
					\item The pair $(V,\msc{F})$ is said to satisfy the \textit{restricted center property} $($\trcp$)$ if for each $F \in \msc{F}$, $\textrm{cent}_{V}(F) \neq \es$. In particular, if $\msc{F}$ is the class of all singleton subsets of $X$ and $(V, \msc{F})$ satisfies \trcp, then $V$ is said to be proximinal in $X$.
					\item The set-valued function, denoted by $\textrm{cent}_{V}(.)$, which maps each $F \in \msc{F}$ to the set $cent_{V}(F)$ is called the restricted Chebyshev-center map on $\msc{F}$. In particular, if $V=X$, then the map $\textrm{cent}_{X}(.)$ is called the \textit{Chebyshev-center} map on $\msc{F}$.
				\end{enumerate}
			\end{Def}
			The continuity of the map $\emph{cent}_{V}(.)$, for a $V \in \mathcal{CV}(X)$, is discussed in the Hausdorff metric. For a Banach space $X$, the Hausdorff metric $d_{H}$ is defined as follows: For each $B_1,B_2 \in \mc{CB}(X)$, \[d_{H}(B_1,B_2) = \inf\{a>0: B_1 \ci B_2 + aB(0,1), B_2 \ci  B_1 + aB(0,1)\}.\]
			
			One of the classical problems in approximation theory is determining the existence of restricted Chebyshev centers in Banach spaces. In this paper, one of our aims is to study this problem in $L_1$-predual spaces. A Banach space $X$ is said to be an $L_{1}$-predual space if the dual space of $X$, denoted by $X^\ast$, is isometric to an $L_{1}(\mu)$ space, for some positive measure $\mu$. The spaces of real-valued continuous functions on a compact Hausdorff space $K$ and those vanishing at infinity on a locally compact Hausdorff space $T$, denoted by $C(K)$ and $C_0(T)$ respectively, are two important examples of $L_1$-predual spaces. One can refer \cite{LL} for more examples of an $L_1$-predual space.
			
			It is known that an $L_1$-predual space $X$ admits centers for $\mc{K}(X)$ (see \cite[Corollary~$4$]{BR} and \cite[Theorem~$4.8$, p.~$38$]{JLin}). In this paper, we use a different approach to not only prove the existence but also give an explicit description of a Chebyshev center of a compact subset of an $L_1$-predual space. This description further aids us to prove a necessary and sufficient condition for the set $\emph{cent}_{V}(F)$ to be non-empty, whenever $X$ is an $L_1$-predual space, $F \in \mc{K}(X)$ and $V \in \mc{CV}(X)$. The motivation for the above condition is \cite[Theorem~2.2]{SW}.
			
			There are various types of characterizations available in the literature for $L_1$-predual spaces; for example, see \cite[Chapter~7, Section~21]{Lacey} and \cite{Rao2}. \cite{EWW} characterized the spaces of the type $C(K)$ and $C_{0}(T)$, whenever $K$ is a compact Hausdorff space and $T$ is a locally compact Hausdorff space, in terms of an identity in the following manner.
			
			\begin{thm}[{\cite{EWW}}]\label{T1}
				Let $X$ be a Banach space. Then the following statements are equivalent.
				\begin{enumerate}[{(i)}]
					\item $X$ is isometric to either a $C(K)$ space or a $C_{0}(T)$ space, for some compact Hausdorff space $K$ or a locally compact Hausdorff space $T$.
					\item For each $V \in \mc{CV}(X)$ and $B \in \mc{CB}(X)$, 
					$\textrm{rad}_{V}(B) = \textrm{rad}_{X}(B) + d(V,\textrm{cent}_{X}(B)).$
					\item For each $V$~$\in$~$\mc{CV}(X)$ and $B$~$\in$~$\mc{CB}(X)$, 
					$\textrm{rad}_{V}(B) = \textrm{rad}_{X}(B) + \lim_{\delta \rightarrow 0^{+}} d(V,\textrm{cent}_{X}(B,\delta)).$
				\end{enumerate}
			\end{thm}
			The above characterization validates our corresponding investigation in $L_1$-predual spaces. In fact, we prove that if the conditions \textit{(ii)} and \textit{(iii)} of Theorem~\ref{T1} are weakened by replacing $\mc{CB}(X)$ with $\mc{F}(X)$, then these conditions characterize an $L_1$-predual space; see Theorem~\ref{T2.3.5}. 
			
			We now recall a notion in Banach spaces called as an $M$-ideal, which is stronger than proximinality.
			\begin{Def}[{\cite{HWW}}]
				Let $X$ be a Banach space. 
				\begin{enumerate}[{(i)}]
					\item A linear projection~$P$ on $X$ is said to be an $L$-projection $(M\mbox{-projection})$ if for each $x \in X$, $\|x\|= \|Px\| +\|x-Px\|$ $(\|x\|= \max\{\|Px\|, \|x-Px\|\})$. 
					\item A subspace~$J$ of $X$ is said to be an $L$-summand $(M\mbox{-summand})$ in $X$ if it is the range of an $L$-projection $(M\mbox{-projection})$.
					\item A subspace~$J$ of $X$ is said to be an $M$-ideal in $X$ if the annihilator of $J$, denoted by $J^{\perp}$, is an $L$-summand in $X^\ast$.
				\end{enumerate}
			\end{Def}
			
			It is easy to deduce from \cite[Theorem~$2.2$, p.~$18$]{HWW} and \cite[Theorem~$6$, p.~$212$]{Lacey} that an $M$-ideal in an $L_1$-predual space is again an $L_1$-predual space. Further, it follows from \cite[Corollary~4]{Amir} that for a compact Hausdorff space $K$, if $J$ is an $M$-ideal in $C(K)$, then $(J,\mc{CB}(C(K)))$ satisfies \ercp. Therefore, motivated by the above result and \cite[Theorem~I.2.2]{Ward}, we ask the following question and answer it completely in section~\ref{sec2.30}. 
			
			\begin{prob}\label{Q3}
				Let $K$ be a compact Hausdorff space. Let $J$ be an $M$-ideal in $C(K)$ and $B \in \mc{CB}(C(K))$. Is it possible to describe the set 
				\begin{equation}\label{EqnQ3.1}
					\textrm{cent}_{J}(B) = \{h \in J : N_{B} -r_B \leq h \leq n_{B} + r_B\},
				\end{equation}
				where for each $t \in K$, $m_{B}(t) = \inf\{b(t): b \in B\}$, $n_{B}(t) = \liminf_{s \rightarrow t} m_{B}(s)$, $M_{B}(t) = \sup\{b(t):b \in B\}$,
				$N_{B}(t) = \limsup_{s \rightarrow t} M_{B}(s)$ and 
				$r_{B} = \frac{1}{2} \sup\{N_{B}(t) - n_{B}(t): t\in K\}$?
			\end{prob}
			
			We now provide a brief account of each section.
			
			In section~\ref{sec2.2}, we discuss a few preliminaries and observations which lay the groundwork for the results in the subsequent sections. 
			
			In section~\ref{sec2}, we present the main results of this paper. It is divided into two subsections. In subsection~\ref{sec2.30}, we prove that the Chebyshev centers of the closed bounded subsets of the $M$-summands in $C(K)$ have the description as in ($\ref{EqnQ3.1}$); see Theorem~\ref{T2.3.9}. We illustrate with examples that in general, the description in ($\ref{EqnQ3.1}$) may not hold true for $M$-ideals which are not $M$-summands.
			
			Let $A(K)$ denote the space of real-valued affine continuous functions on a compact convex subset $K$ of a locally convex topological vector space (\textit{lctvs}). In subsection~\ref{sec2.3}, for an $L_1$-predual space $X$ and $F \in \mc{K}(X)$, we explicitly describe the elements in $\emph{cent}_{X}(F)$ via a well-known isometric identification of $X$ to a subspace of $A(B_{X^\ast})$, where $B_{X^{\ast}}$ is endowed with the weak$^\ast$ topology (see Proposition~\ref{L2.2.6}). This description leads us to the following consequences which are of interest. We prove that for an $L_1$-predual space $X$, $F \in \mc{K}(X)$ and $V \in \mc{CV}(X)$, the set $\emph{cent}_{V}(F) \neq \es$ if and only if the infimum defining $d(V, \emph{cent}_{X}(F))$ is attained. Furthermore, we characterize an $L_1$-predual space $X$ as a Banach space which satisfies the identity $\emph{rad}_{V} (F) = \emph{rad}_{X}(F) + d(V,\emph{cent}_{X}(F))$, for each $V \in \mc{CV}(X)$ and $F \in \mc{K}(X)$. We also prove that for an $L_1$-predual $X$, the map $\emph{cent}_{X}(.)$ is $2$-Lipschitz continuous on $\mc{K}(X)$ in the Hausdorff metric. 
			
			\section{Preliminaries}\label{sec2.2}
			Let $K$ be a compact Hausdorff space and $\mathcal{A}$ be a subspace of $C(K)$. Let $B \in \mc{CB}(\mc{A})$. 
			We now define the following functions. For each $t \in K$,
			\begin{equation}\label{Eq2.2.1}
				\begin{split}
					m_{B}(t) &= \inf\{b(t): b \in B\},\\
					n_{B}(t) &= \liminf_{s \rightarrow t} m_{B}(s),\\
				\end{split}
				\quad \quad
				\begin{split}
					M_{B}(t) &= \sup\{b(t):b \in B\},\\
					N_{B}(t) &= \limsup_{s \rightarrow t} M_{B}(s).
				\end{split}
			\end{equation}
			Let us also define 
			\begin{equation}\label{Eq2.2.2}
				r_{B} = \frac{1}{2} \sup\{N_{B}(t) -n_{B}(t): t \in K\}.
			\end{equation} 
			
			\begin{rem}\label{R2.2.3}
				The following remarks are few properties of the functions and number defined in $(\ref{Eq2.2.1})$ and $(\ref{Eq2.2.2})$ respectively, which are easy to verify.
				\begin{enumerate}[{(i)}]
					\item If $B \in \mc{CB}(\mc{A})$, then the functions $N_{B}, m_{B}$ are upper semicontinuous and  $n_{B}, M_{B}$ are lower semicontinuous functions on $K$.
					\item If $F \in \mc{K}(\mc{A})$, then the functions $M_{F}$ and $m_{F}$ are continuous on $K$ and consequently, $N_{F}=M_{F}$ and $n_{F}=m_{F}$.
					\item For each $B \in \mc{CB}(\mc{A})$, $r_{B} \leq \textrm{rad}_{\mathcal{A}}(B)$.~A proof of this inequality can be found in \cite[Lemma~I.2.1]{Ward}.
				\end{enumerate}
			\end{rem}
			
			For a Banach space $X$ and a bounded subset $A$ of $X$, let $\emph{diam}(A) = \sup \{\|a-a^\prime\|: a,a^\prime \in A\}$.
			\begin{lem}\label{L2.2.4}
				Let $K$ be a compact Hausdorff space and $\mathcal{A}$ be a subspace of $C(K)$. If $F \in \mc{K}(\mc{A})$, then $\frac{1}{2} \textrm{diam}(F) \leq r_{F}$. Moreover, if $F \in \mc{F}(\mc{A})$, then the equality holds.  
			\end{lem}
			\begin{proof}
				Let $F \in \mc{K}(\mc{A})$. Then from Remark~\ref{R2.2.3} \textit{(ii)}, $n_{F} = m_{F}$ and $N_{F} = M_{F}$. Let $L=\sup\{M_{F}(t)-m_{F}(t):t \in K\}.$ For each $t \in K$, 
				\begin{equation}
					M_{F}(t) -m_{F}(t) = \sup_{z \in F} z(t) - \inf_{z \in F} z(t)=\sup_{z \in F} z(t) +\sup_{z \in F} -z(t).
				\end{equation}
				Now, for each $z_1, z_2 \in F$ and $t \in K$, \[\pm(z_1(t)-z_2(t)) \leq \sup_{z^\prime \in F} z^\prime(t) +\sup_{z^\prime \in F} -z^\prime(t).\] It follows that $\|z_1-z_2\| \leq L$ and hence, $\emph{diam}(F) \leq L$. This gives us the desired inequality.
				
				Now, let $F \in \mc{F}(\mc{A})$. Then, for each $t \in K$, there exists $z_1, z_2 \in F$ such that 
				\begin{equation}\label{eqn6}
					\begin{split}
						M_{F}(t) -m_{F}(t) = \max_{z^\prime \in F} z^\prime(t) - \min_{z^\prime \in F} z^\prime(t) = z_1(t) -z_2(t).
					\end{split}
				\end{equation}
				This implies $M_{F}(t) -m_{F}(t) \leq \|z_1 -z_2\| \leq \emph{diam}(F).$ Hence $L \leq \emph{diam}(F)$. 
			\end{proof}	
			
			For a subspace $\mathcal{A}$ of $C(K)$ and $B \in \mc{CB}(\mathcal{A})$, we provide a description of $\emph{cent}_{\mathcal{A}} (B)$ in the following result.
			
			\begin{prop}\label{P2.2.5}
				Let $K$ be a compact Hausdorff space and $\mathcal{A}$ be a subspace of $C(K)$. If $B \in \mc{CB}(\mc{A})$ such that $\textrm{cent}_{\mc{A}}(B) \neq \es$, then \[\textrm{cent}_{\mc{A}}(B) = \{x \in \mc{A}: N_{B}-\textrm{rad}_{\mc{A}}(B) \leq x \leq n_{B}+ \textrm{rad}_{\mc{A}}(B)\}.\]
			\end{prop}
			\begin{proof}
				Let $B \in \mc{CB}(\mc{A})$. Suppose $x \in \emph{cent}_{\mc{A}}(B)$. Then $\emph{rad}_{\mc{A}}(B) = r(x,B).$ It follows that for each $t \in K$ and $b \in B$, \[b(t) - \emph{rad}_{\mc{A}}(B) \leq x(t) \leq b(t)+\emph{rad}_{\mc{A}}(B).\] Then, from the definitions in $(\ref{Eq2.2.1})$, for each $t \in K$,
				\begin{equation}\label{E5}
					N_{B}(t) -\emph{rad}_{\mc{A}}(B) \leq x(t) \leq n_{B}(t)+\emph{rad}_{\mc{A}}(B).
				\end{equation} Now, if $x \in \mc{A}$ such that the inequalities in $(\ref{E5})$ hold, then using the fact that $M_{B}$ and $-m_{B}$ are lower semicontinuous on $K$, it is easy to deduce that $\emph{rad}_{\mc{A}}(B) =r(x,B)$.
			\end{proof}			
			
			Let $X$ be a Banach space. Let $B_{X^{\ast}}$ be endowed with the weak$^\ast$ topology. There exists a natural affine homeomorphism from $B_{X^\ast}$ onto $B_{X^\ast}$ given by the map $\sigma(x^\ast) = -x^\ast$, for each $x^\ast \in B_{X^\ast}$. It is easy to see that $\sigma^2$ is the identity map on $B_{X^\ast}$. Let us define $A_{\sigma}(B_{X^\ast}) = \{a \in A(B_{X^\ast}): a = -a \circ \sigma\}.$ Clearly, $A_{\sigma}(B_{X^\ast})$ is a closed subspace of $A(B_{X^\ast})$. The following result  identifies $X$ with $A_{\sigma}(B_{X^\ast})$.
			
			\begin{lem}[{\cite[Lemma~8, p.~213]{Lacey}}]\label{L2.2.6}
				Let $X$ be a Banach space. Then $X$ is isometric and linearly isomorphic to $A_\sigma(B_{X^\ast})$ under the mapping $x \mapsto \overline{x}$ where $\overline{x}(x^\ast)= x^\ast(x)$, for each $x^\ast \in B_{X^\ast}$ and $x \in X$.
			\end{lem}
			
			\begin{rem}\label{R2.2.7}
				Let $X$ be a Banach space and $F \in \mc{K}(X)$. By Lemma~\ref{L2.2.6}, $F$ can be viewed as a compact subset of $A_\sigma(B_{X^\ast})$. Hence, due to the compactness of $F$, the supremum and infimum defining the functions $M_{F}$ and $m_{F}$ respectively are attained in $B_{X^\ast}$. In fact, for each $x^\ast \in B_{X^\ast}$, 
				\begin{equation}\label{Eq2.2.8}
					\begin{split}
						&M_{F}(x^\ast) = \max \{x^\ast(x): x \in F\}, \\ &m_{F}(x^\ast) = \min \{x^\ast(x): x \in F\}\mbox{ and }\\ &r_{F} = \frac{1}{2} \max \{M_{F}(x^\ast) -m_{F}(x^\ast):x^\ast \in B_{X^\ast}\}.\\
					\end{split}
				\end{equation}
				It is also clear from the latter part of the proof of Lemma~\ref{L2.2.4} and $(\ref{Eq2.2.8})$ that $r_{F} = \frac{1}{2} \textrm{diam}(F)$.
			\end{rem}
			
			We now recall a separation theorem in an $L_1$-predual space, which is proved by Ka-Sing Lau.
			\begin{thm}[{\cite[Theorem~2.3]{Lau}}]\label{T2.2.8}
				Let $X$ be an $L_1$-predual space. If $f$ is a weak$^\ast$-lower semicontinuous concave function on $B_{X^\ast}$ such that for each $x^\ast \in B_{X^\ast}$, $\frac{1}{2}(f(x^\ast)+f(-x^\ast)) \geq 0$, then there exists $a \in A_{\sigma}(B_{X^\ast})$ such that $f \geq a$.
			\end{thm}
			
			Now, for a Banach space $X$, if $V \in \mc{CV}(X)$, then the following result tells us that the map $B \mapsto \emph{rad}_{V}(B)$, for $B \in \mc{CB}(X)$ is continuous on $\mc{CB}(X)$ in the Hausdorff metric. The following result is proved in \cite[Theorem~2.5]{SoTa}. We include the proof here for the sake of completeness. 
			
			\begin{lem}\label{L2.2.9}
				Let $X$ be a Banach space, $V \in \mc{CV}(X)$ and $B_{1},B_{2} \in \mc{CB}(X)$. Then for each $v \in V$, $\vert r(v,B_{1})-r(v,B_{2})\vert \leq d_{H}(B_{1},B_{2})$ and $\vert\textrm{rad}_{V}(B_{1})-\textrm{rad}_{V}(B_{2})\vert \leq d_{H}(B_{1},B_{2}).$
			\end{lem}
			\begin{proof}
				Let $v \in V$. Now, let $b_1 \in B_{1}$ and $\varepsilon>0$. Choose $b_2 \in B_{2}$ such that $\|b_1-b_2\|<d_{H}(B_{1},B_{2})+\varepsilon$. Then \[\|v-b_1\| \leq \|v-b_2\|+\|b_2-b_1\|<r(v,B_{2})+d_{H}(B_{1},B_{2})+\varepsilon.\] Since $\varepsilon$ is arbitrary, it follows that 
				\begin{equation}\label{Eq2.2.10}
					r(v,B_{1}) \leq r(v,B_{2})+d_{H}(B_{1},B_{2}).
				\end{equation} 
				Further, after swapping $B_{1}$ with $B_{2}$ in the above argument, we obtain the following inequality. 
				\begin{equation}\label{Eq2.2.11}
					r(v,B_{2}) \leq r(v,B_{1})+d_{H}(B_{1},B_{2}).
				\end{equation} 
				The first conclusion of the result follows from the inequalities in $(\ref{Eq2.2.10})$ and $(\ref{Eq2.2.11})$.
				
				The inequalities in $(\ref{Eq2.2.10})$ and $(\ref{Eq2.2.11})$ hold true for every $v \in V$ and hence, the final conclusion of the result follows. 
			\end{proof}
			
			Let us recall the definition of a convex symmetric lower semicontinuous set-valued function.
			\begin{Def}[\cite{LL}]\label{D2.2.12}
				Let $C$ be a convex subset of an \textit{lctvs} and $E$ be another \textit{lctvs}. Let $\Phi$ be a set-valued map from $C$ to the family of non-empty convex subsets of $E$.
				The map $\Phi$ is said to be
				\begin{enumerate}[{(i)}]
					\item a {\it convex function }on $C$ if for each $c \in C$, $\Phi(c)$ is a non-empty convex set and for each $0 < \al<1$ and $c_1,c_2 \in C$, \[\al \Phi(c_1)+(1-\al)\Phi(c_2) \ci \Phi(\al c_1 + (1-\al)c_2);\]
					\item a {\it lower semicontinuous function} on $C$ if for each open subset $U$ of $E$, the set $\{c \in C: \Phi(c) \cap U \neq \es\}$ is a relatively open subset of $C$ and
					\item a {\it symmetric function }on $C$ if for each $c,-c \in C$, $\Phi(-c) = -\Phi(c)$.
				\end{enumerate}
			\end{Def}
			
			We conclude this section by recalling a selection theorem by Lazar and Lindenstrauss. We state it according to our purpose in the following manner.
			\begin{thm}[{\cite[Theorem~2.2]{LL}}]\label{T2.2.80}
				Let $X$ be an $L_1$-predual space and $Y$ be another Banach space. Let $\Phi: B_{X^\ast} \rightarrow \mc{CV}(Y)$ be a convex symmetric weak$^\ast$-lower semicontinuous set-valued function on $B_{X^\ast}$. Then $\Phi$ admits a selection from $A_{\sigma}(B_{X^\ast})$, that is, there exists $a \in A_{\sigma}(B_{X^\ast})$ such that for each $x^\ast \in B_{X^\ast}$, $a(x^\ast) \in \Phi(x^\ast)$.
			\end{thm} 
			
			\section{Main Results}\label{sec2}
			\subsection{Chebyshev centers in $M$-ideals in $C(K)$}\label{sec2.30}
			Let $K$ be a compact Hausdorff space. We use the following notation for $M$-ideals in $C(K)$ in the discussion that follows. Consider a subset $D$ of $K$ and define $J_{D} = \{h \in C(K):h(t) = 0\mbox{, for each }t \in D \}.$ From \cite[Example~$1.4$~(a), p.~$3$]{HWW}, $J$ is an $M$-ideal in $C(K)$ if and only if there exists a closed subset $D$ of $K$ such that $J=J_{D}$ and $J$ is an $M$-summand in $C(K)$ if and only if there exists a clopen subset (in other words, a subset which is open as well as closed) $D$ of $K$ such that $J=J_{D}$. 
			
			We now make the following easy observation.
			\begin{prop}\label{P2.3.7}
				Let $K$ be a compact Hausdorff space and $D$ be a closed subset of $K$. Then the subspace $J_{D}$ of $C(K)$ admits centers for $\mc{K}(J_{D})$ and for each $F\in \mc{K}(J_{D})$,
				\begin{equation}
					\textrm{cent}_{J_{D}}(F) = \{h \in J_{D}: M_{F}-r_{F} \leq h \leq m_{F}+r_{F}\}.
				\end{equation} 
				Furthermore, $\textrm{rad}_{J_{D}}(F)=r_{F}$.
			\end{prop}
			\begin{proof}
				Let $F \in \mc{K}(J_{D})$. From Remark~\ref{R2.2.3} \textit{(ii)}, $M_F$ and $m_F$ are continuous functions on $K$. Clearly, for each $t \in D$, $M_{F}(t) = 0 = m_{F}(t).$ Therefore, $\frac{M_{F}+m_{F}}{2} \in \{h \in J_{D}: M_{F}-r_{F} \leq h \leq m_{F}+r_{F}\}$. It follows that $\emph{rad}_{J_{D}}(F) \leq r_{F}$. Now, from Remark~\ref{R2.2.3} \textit{(iii)} and Proposition~\ref{P2.2.5}, it follows that $\emph{rad}_{J_{D}}(F)=r_{F}$ and \[\emph{cent}_{J_{D}}(F) = \{h \in J_{D}: M_{F}-r_{F} \leq h \leq m_{F}+r_{F}\}.\]
			\end{proof}	
			
			Before proceeding to the next theorem, we need the following technical result. It is a variant of the classical insertion type theorem.
			
			\begin{lem}\label{L2.3.8}
				Let $D$ be a clopen subset of a compact Hausdorff space $K$. If $g,-f$ are upper semicontinuous functions on $K$ such that $g \leq f$ on $K$ and for each $t \in D$, $g(t) \leq \al \leq f(t)$, then there exists $h \in C(K)$ such that $g\leq h \leq f$ on $K$ and for each $t \in D$, $h(t) =\alpha$.
			\end{lem}
			
			\begin{proof}
				Without loss of generality, we assume that $D$ is a non-empty proper subset of $K$. Let $r > s$ be such that \[r \geq \sup \{g(t), \al: t \in K\}\quad\mbox{and}\quad s \leq \inf \{f(t), \al: t \in K\}.\]
				Let us define the functions $h_r$ and $h_s$ on $K$ in the following manner.
				\begin{equation*}
					h_r(t)=
					\begin{dcases}
						r, & \mbox{if }t \in K \backslash D; \\
						\al, & \mbox{if }t\in D. 
					\end{dcases}
					\quad\mbox{and}\quad
					h_s(t)=
					\begin{dcases}
						s, & \mbox{if }t \in K \backslash D; \\
						\al, & \mbox{if }t \in D. 
					\end{dcases}	
				\end{equation*}
				Now, let $t_0 \in K$. Then $t_{0} \in D$ or $t_0 \in K \backslash D$. Since $D$ is clopen, in each of above cases, it is easy to see that $\limsup_{t \rightarrow t_0} h_s(t) \leq h_{s}(t_0).$ Therefore, $h_s$ is an upper semicontinuous function on $K$. Similarly, we can prove that $-h_r$ is also an upper semicontinuous function on $K$. Hence, $-\min \{h_r,f\}$ and $\max \{h_s, g\}$ are also upper semicontinuous functions on $K$ such that $\max \{h_s, g\} \leq \min \{h_r,f\}.$ Therefore, by Kat\v{e}tov's insertion theorem (see \cite{Katetov}), there exists $h \in C(K)$ such that $\max \{h_s, g\} \leq h \leq \min \{h_r,f\}.$ Further, for each $t \in D$, \[\max \{h_s(t), g(t)\} = \al \leq  h(t) \leq \min \{h_r(t),f(t)\} = \al.\] This implies for each $t \in D$, $h(t)=\al$. 
			\end{proof}	
			
			\begin{thm}\label{T2.3.9}
				Let $K$ be a compact Hausdorff space and $J$ be an $M$-summand in $C(K)$. Then $J$ admits centers for $\mc{CB}(J)$ and for each $B \in \mc{CB}(J)$, 
				\begin{equation}\label{eqn6.0}
					\textrm{cent}_{J}(B) = \{h \in J: N_{B}-r_{B} \leq h \leq n_{B}+r_{B}\}.
				\end{equation} 
				Furthermore, $\textrm{rad}_{J}(B)=r_{B}$.
			\end{thm}
			\begin{proof}
				Since $J$ is an $M$-summand in $C(K)$, there exists a clopen subset $D$ of $K$ such that $J= J_D$. Let $B \in \mc{CB}(J_{D})$. Clearly, for each $t \in D$, $M_{B}(t) = 0 = m_{B}(t).$ Using the assumption that $D$ is clopen in $K$ and by the definitions of the functions $N_{B}$ and $-n_{B}$, it follows that for each $t \in D$, $N_{B}(t) =0=n_{B}(t)$. Hence for each $t \in D$, \[N_{B}(t)-r_{B} \leq 0 \leq n_{B}(t)+r_{B}.\] By Remark~\ref{R2.2.3} \textit{(i)} and Lemma~\ref{L2.3.8}, there exists $h \in J_D$ such that $N_{B}-r_{B} \leq h \leq n_{B}+r_{B}$. This shows that the set on the right-hand side of $($\ref{eqn6.0}$)$ is non-empty. Hence, it follows that $\emph{rad}_{J_{D}}(B) \leq r_{B}$. Therefore, from Remark~\ref{R2.2.3} \textit{(iii)} and Proposition~\ref{P2.2.5}, it follows that $\emph{rad}_{J_{D}}(B)=r_{B}$ and \[\emph{cent}_{J_{D}}(B)=\{h \in J_{D}: N_{B}-r_{B} \leq h \leq n_{B}+r_{B}\}.\]
			\end{proof}
			
			Consider the notations and hypotheses of Theorem~\ref{T2.3.9}. For a $B \in \mc{CB}(C(K))$, even though $\emph{cent}_{J_{D}}(B) \neq \es$, it is not necessary that the set $\emph{cent}_{J_{D}}(B)$ has a description as given in $(\ref{eqn6.0})$. It is possible to construct many examples to illustrate this fact. We provide one such example below.
			
			\begin{ex}\label{Ex2.3.11}
				Consider the subspace $J = \{h \in C(\{0,1\}): h(0)=0\}$ of $C(\{0,1\})$. Define the functions $f,g: \{0,1\} \rightarrow \mb{R}$ as $f(0) = 2$, $f(1) = 0$, $g(0)=3$ and $g(1) = 1$.	Let $B = \{f,g\} \ci C(\{0,1\}) - J$. Clearly, $M_{B} = g$, $m_{B} = f$ and $r_{B} = \frac{1}{2}$. If $h \in C(\{0,1\})$ such that $M_{B}-r_{B} \leq h \leq m_{B}+r_{B}$, then $h(0) = \frac{5}{2}$ and $h(1) = \frac{1}{2}$ and hence, $h \not\in J$.
			\end{ex}
			
			Further, we also remark that if the set $D$ is not clopen as given in Theorem~\ref{T2.3.9}, then the set $\emph{cent}_{J_{D}}(B)$ need not have the description as given in $(\ref{eqn6.0})$ for each $B \in \mc{CB}(J_{D})$. The following example supports this fact.
			
			\begin{ex}\label{Ex2.3.12}
				Let $0< a < b <1$. Consider the space $C([0,1])$ and $D = [a,b] \ci [0,1]$. Let \[B = \left\{f \in C([0,1]): f\left([a,b]\right)=0\mbox{ and }0\leq f(t) \leq 1\mbox{, for all }t\right\}\] Then, cleary, $B \in \mc{CB}(J_{D}) - \mc{K}(J_{D})$ and for each $t \in [0,1] - [a,b]$, $M_{B}(t) = 1$; $M_{B}([a,b]) =0$ and for each $t \in [0,1]$, $m_{B}(t)=0$. Therefore, for each $t \in [0,1]-(a,b)$, $N_{B}(t) = 1$ and $n_{B}(t) = 0$ and for each $t \in (a,b)$, $N_{B}(t)=0=n_{B}(t)$. It follows that $r_{B} = \frac{1}{2}$. If $h \in C([0,1])$ such that $N_{B}- r_{B} \leq h \leq n_{B} + r_B$, then $h(t) = \frac{1}{2}$, for each $t \in [0,1]-(a,b)$ and therefore, $h \not \in J_{D}$. 
			\end{ex}
			
			\subsection{Restricted Chebyshev centers in an $L_1$-predual space}\label{sec2.3}
			We begin this section by proving the existence and providing a description, which is similar to that in $(\ref{eqn6.0})$, of a Chebyshev center of a compact subset of an $L_1$-predual space, with the help of the isometric identification in Lemma~\ref{L2.2.6}.
			\begin{thm}\label{T2.3.1}
				Let $X$ be an $L_1$-predual space and $F \in \mc{K}(A_{\sigma}(B_{X^\ast}))$. Then $\textrm{cent}_{A_{\sigma}(B_{X^\ast})}(F) \neq \es$ and 
				\begin{equation}\label{EqT2.3.1}
					\textrm{cent}_{A_{\sigma}(B_{X^\ast})}(F)=\{a \in A_{\sigma}(B_{X^\ast}):M_{F}-r_{F} \leq a \leq m_{F}+r_{F}\}.
				\end{equation}
				Furthermore, $\textrm{rad}_{A_{\sigma}(B_{X^\ast})}(F)= r_{F}$.
			\end{thm}
			\begin{proof}
				Let $F \in \mc{K}(A_{\sigma}(B_{X^\ast}))$. 
				Let us define $f = m_{F} +r_{F}$ on $B_{X^\ast}$. It follows from Remark~\ref{R2.2.7} that $f$ is a weak$^\ast$-continuous concave function on $B_{X^\ast}$. Let $x^{\ast} \in B_{X^\ast}$. Then, 
				\begin{equation}
					\begin{split}
						&\frac{1}{2} (f(x^\ast) + f(-x^\ast))\\&=\frac{1}{2} (m_{F}(x^\ast) +r_{F} + m_{F}(-x^\ast) + r_{F} ) \\&=\frac{1}{2} [(m_{F}(x^\ast) +r_{F}) - (M_{F}(x^\ast) - r_{F})] \geq 0.
					\end{split}
				\end{equation}
				Therefore, by Theorem~\ref{T2.2.8}, there exists $a \in A_{\sigma}(B_{X^\ast})$ such that $m_{F} + r_{F} \geq a$. Now, for each $x^\ast \in B_{X^\ast}$, $m_{F}(-x^{\ast}) + r_{F} \geq a(-x^\ast)$ and hence, $M_{F}(x^\ast) - r_{F} \leq a(x^\ast)$. Therefore, $M_{F}-r_F \leq a\leq m_F + r_F$ on $B_{X^\ast}$. It follows that $\emph{rad}_{A_{\sigma}(B_{X^\ast})}(F)= r_{F}$. Therefore, from Proposition~\ref{P2.2.5}, we can deduce that $\emph{cent}_{A_{\sigma}(B_{X^\ast})}(F) \neq \es$ and \[\emph{cent}_{A_{\sigma}(B_{X^\ast})}(F)=\{a \in A_{\sigma}(B_{X^\ast}):M_{F}-r_{F} \leq a \leq m_{F}+r_{F}\}.\]
			\end{proof}	
			
			\begin{rem}
				Let $X$ be an $L_1$-predual space and $F \in \mc{K}(X)$. Using Lemma~\ref{L2.2.6}, we can view $F$ as a compact subset of $A_{\sigma}(B_{X^\ast})$. Therefore, applying Theorem~\ref{T2.3.1}, it is easy to see that $\textrm{cent}_{X}(F) \neq \es$ and $\textrm{rad}_{X}(F) = r_{F}$. Moreover, the image of the set $\textrm{cent}_{X}(F)$ under the mapping given in Lemma~\ref{L2.2.6} is precisely the set $\textrm{cent}_{A_{\sigma}(B_{X^\ast})}(F)$. 
			\end{rem}
			
			Before proceeding to the next result, let us recall the definition of a Choquet simplex. A compact convex subset $K$ of a \textit{lctvs} is called a Choquet simplex if $A(K)^{\ast}$ is a lattice. We  refer \cite{asimow} for a detailed study on Choquet simplex. 
			
			\begin{cor}\label{C2.3.2}
				Let $K$ be a Choquet simplex. Then for each $F \in \mc{K}(A(K))$, $\textrm{cent}_{A(K)}(F) \neq \es$ and \[\textrm{cent}_{A(K)}(F) = \{a \in A(K): M_{F}-r_{F} \leq a \leq m_{F}+r_{F}\}.\]
			\end{cor}
			\begin{proof}
				From \cite[Theorem~2, p.~185]{Lacey}, $A(K)$ is an $L_1$-predual space. Now, the result follows from Theorem~\ref{T2.3.1} and \cite[Theorem~4.7, p.~14]{asimow}.
			\end{proof}

			\begin{rem}
				The result in Corollary~\ref{C2.3.2} can also be proved using the Edwards' separation theorem (see \cite[Theorem~7.6, p.~72]{asimow}). Indeed, for a Choquet simplex $K$, if $F \in \mc{K}(A(K))$, then $M_F-r_F$ and $-(m_F + r_F)$ are continuous convex functions on $K$. Hence by Edwards' separation theorem, there exists $a \in A(K)$ such that $M_F-r_F \leq a \leq m_F + r_F$. It follows that $\textrm{rad}_{A(K)}(F) = r_F$. The desired conclusion now follows from Proposition~\ref{P2.2.5}.
			\end{rem}
			
			\begin{thm}\label{T2.3.3}
				Let $X$ be an $L_1$-predual space. Then the map $\textrm{cent}_{X}(.)$ is $2$-Lipschitz continuous on $\mc{K}(X)$ in the Hausdorff metric and in general, the constant $2$ is the best possible Lipschitz constant.
			\end{thm}
			\begin{proof}
				From Lemma~\ref{L2.2.6}, it suffices to prove that the map $\emph{cent}_{A_{\sigma}(B_{X^\ast})}(.)$ is $2$-Lipschitz continuous on $\mc{K}(A_{\sigma}(B_{X^\ast}))$ in the Hausdorff metric. 
				
				Let $\varepsilon>0$ and $F_1,F_2 \in \mc{K}(A_{\sigma}(B_{X^\ast}))$ be such that $d_{H}(F_1 , F_2)<\frac{\varepsilon}{2}.$ This implies there exists $0<\delta<\frac{\varepsilon}{2}$ such that 
				\begin{equation}\label{Eq2.3.0}
					F_1 \ci F_2 + \delta B(0,1)\mbox{ and }F_2 \ci  F_1 + \delta B(0,1).
				\end{equation}
				For each $i=1,2$, let $r_{i} = \emph{rad}_{A_{\sigma}(B_{X^\ast})}(F_{i})$. Then by Theorem~\ref{T2.3.1}, for each $i=1,2$, 
				\begin{equation}\label{Eq2.3.00}
					\emph{cent}_{A_{\sigma}(B_{X^\ast})}(F_{i}) = \{a \in A_{\sigma}(B_{X^\ast}):M_{F_{i}}-r_i \leq a \leq m_{F_{i}} +r_i\}.
				\end{equation}
				and by Lemma~\ref{L2.2.9}, $r_1-\delta < r_2 < r_1 + \delta.$
				
				Let $a \in \emph{cent}_{A_{\sigma}(B_{X^\ast})}(F_1)$ and $z_{2} \in F_{2}$. Then it follows from $(\ref{Eq2.3.0})$ that there exists $z_{1} \in F_{1}$ and $z_0 \in B(0,1)$ such that $z_{2} = z_{1} + \delta z_0$. Therefore, using $($\ref{Eq2.3.00}$)$, it follows that for each $x^\ast \in B_{X^\ast}$,
				\begin{equation}\label{ET4.1.1}
					\begin{split}
						z_{2}(x^\ast) -r_2 &= z_{1}(x^\ast) + \delta z_{0}(x^\ast) -r_2\\ &< z_{1}(x^\ast) + \delta -r_1 +\delta\\ &= z_{1}(x^\ast)-r_1 +2 \delta\\ &\leq a(x^\ast) + 2 \delta.		     		
					\end{split}
				\end{equation}
				\begin{equation}\label{ET4.1.2}
					\begin{split}
						z_{2}(x^\ast) +r_2 &= z_{1}(x^\ast) + \delta z_{0}(x^\ast) +r_2\\ &> z_{1}(x^\ast) - \delta +r_1 -\delta\\ &=z_{1}(x^\ast) + r_1 - 2 \delta\\ &\geq a(x^\ast) - 2\delta.		     		
					\end{split}
				\end{equation}
				
				It follows from $(\ref{ET4.1.1})$, $(\ref{ET4.1.2})$ and the definitions of $M_{F_{2}}$ and $m_{F_{2}}$ that $M_{F_{2}} -r_{2} \leq a+2\delta $ and $a-2\delta \leq m_{F_{2}} + r_{2}$ on $B_{X^\ast}$. Define $g= \max \{M_{F_{2}} -r_{2},a-2\delta\}$ and $f =\min\{a+2 \delta, m_{F_{2}} + r_{2}\}$. Clearly, $g \leq f$ and $-g,f$ are weak$^\ast$-continuous concave functions on $B_{X^\ast}$. For each $x^\ast \in B_{X^\ast}$, by using the fact that  $M_{F_2}(-x^\ast) = -m_{F_2}(x^\ast)$, it is easy to verify that  \[\frac{1}{2}(f(x^\ast)+f(-x^\ast))\geq 0.\] 
				Therefore, by Theorem~\ref{T2.2.8}, there exists $a^\prime \in A_{\sigma}(B_{X^\ast})$ such that $f \geq a^\prime$. Now, for each $x^\ast \in B_{X^\ast}$, since $f(x^\ast) = -g(-x^\ast)$, we have $g(-x^\ast) \leq -a^{\prime}(x^\ast) = a^{\prime}(-x^\ast)$. It follows that $g \leq a^\prime \leq f$ on $B_{X^\ast}$. Hence, $a^\prime \in \emph{cent}_{A_{\sigma}(B_{X^\ast})}(F_2)$ such that $\|a^\prime - a\|\leq 2 \delta < \varepsilon$. Therefore, this proves that $\emph{cent}_{A_{\sigma}(B_{X^\ast})}(F_1) \ci cent_{A_{\sigma}(B_{X^\ast})}(F_2) +\varepsilon B(0,1)$. Similarly, it can be proved that $\emph{cent}_{A_{\sigma}(B_{X^\ast})}(F_2) \ci \emph{cent}_{A_{\sigma}(B_{X^\ast})}(F_1) +\varepsilon B(0,1)$.
				Therefore, $d_{H}(\emph{cent}_{A_{\sigma}(B_{X^\ast})}(F_1),\emph{cent}_{A_{\sigma}(B_{X^\ast})}(F_2)) \leq\varepsilon.$ Since $\varepsilon>0$ is arbitrary, $d_{H}(\emph{cent}_{A_{\sigma}(B_{X^\ast})}(F_1),\emph{cent}_{A_{\sigma}(B_{X^\ast})}(F_2)) \leq 2 d_{H}(F_1, F_2)$.
				
				In order to show that the constant $2$ is the best possible choice, consider $\mb{R}^2$ equipped with the supremum norm. Let $F = \{(-1,0), (1,0)\}$ and $G = \{(0,1)\}$. Clearly, $d_{H}(F,G) = 1$. Further, $\emph{cent}_{\mb{R}^2}(F) = \{(0,\lambda) : -1 \leq \lambda \leq 1\}$ and $\emph{cent}_{\mb{R}^{2}}(G) = \{(0,1)\}$. Then clearly, $d_{H}(\emph{cent}_{\mb{R}^2}(F),\emph{cent}_{\mb{R}^2}(G)) = 2$.
			\end{proof} 
			
			For a Banach space $X$, let $C(K,X)$ denote the Banach space of $X$-valued continuous functions on a compact Hausdorff space $K$. If $X$ is an $L_1$-predual space, then by \cite[Proposition~4.9]{BD}, $C(K,X)$ is also an $L_1$-predual space, for each compact Hausdorff space $K$. Therefore, the following result follows immediately from Theorem~\ref{T2.3.1} and Theorem~\ref{T2.3.3}.
			\begin{cor}
				Let $X$ be an $L_1$-predual space and $K$ be a compact Hausdorff space. Then $C(K,X)$ admits centers for $\mc{K}(C(K,X))$ and the map $\textrm{cent}_{C(K,X)}(.)$ is $2$-Lipschitz continuous on $\mc{K}(C(K,X))$ in the Hausdorff metric.
			\end{cor}
			
			We now prove our next main result.
			\begin{thm}\label{T2.3.4}
				Let $X$ be an $L_1$-predual space. Let $F \in \mc{K}(X)$ and $V \in \mathcal{CV} (X)$. Then,
				the equality, $\textrm{rad}_{V}(F) = \textrm{rad}_{X}(F) + d(V,\textrm{cent}_{X}(F))$, holds true. Furthemore, a necessary and sufficient condition for the set $\textrm{cent}_{V}(F) \neq \es$ is that there exists $v_{0} \in V$ and $x_{0} \in \textrm{cent}_{X}(F)$ such that $d(V,\textrm{cent}_{X}(F)) = \|v_{0} - x_{0}\|$.
			\end{thm}
			\begin{proof}
				By Lemma~\ref{L2.2.6}, we can view $F$ as a compact subset of $A_{\sigma}(B_{X^\ast})$ and $V$ as a closed convex subset of $A_{\sigma}(B_{X^\ast})$. Hence, it suffices to prove the following claims.
				
				{\sc Claim 1:} The following formula is satisfied.  
				\begin{equation}\label{eqn5.1}
					\emph{rad}_{V}(F) = \emph{rad}_{A_{\sigma}(B_{X^\ast})}(F) + d(V,\emph{cent}_{A_{\sigma}(B_{X^\ast})}(F)).
				\end{equation}
				
				{\sc Claim 2:} The set $\emph{cent}_{V}(F) \neq \es$ if and only if $d(V,\emph{cent}_{A_{\sigma}(B_{X^\ast})}(F)) = \|v_{0} - a_{0}\|,$ for some $v_{0} \in V$ and $a_{0} \in \emph{cent}_{A_{\sigma}(B_{X^\ast})}(F)$. 
				
				From Theorem~\ref{T2.3.1}, $\emph{cent}_{A_{\sigma}(B_{X^\ast})}(F) \neq \es$ and $\emph{rad}_{A_{\sigma}(B_{X^\ast})}(F) = r_{F}$. Let $R=d(V,\emph{cent}_{A_{\sigma}(B_{X^\ast})}(F))$. In order to prove the formula in $(\ref{eqn5.1})$, it suffices to show that $\emph{rad}_{V}(F) \geq  r_{F} + R$ because the reverse inequality is a simple consequence of triangle inequality. We implement the proof techniques used in \cite[Theorem~2.2]{SW}.
				
				Let $S_{n} = \emph{rad}_{V}(F) - r_{F} + \frac{1}{n}$, for $n=1,2,\ldots$. There exists $v_{n} \in V$ such that \[r(v_n,F) < \emph{rad}_{V}(F) + \frac{1}{n}.\] Therefore, from the above inequality and $(\ref{Eq2.2.8})$, for each $x^\ast \in B_{X^\ast}$,
				\begin{equation}\label{eqn5.2}
					\begin{split}
						M_{F}(x^\ast) - \emph{rad}_{V}(F) - \frac{1}{n} \leq v_{n}(x^\ast) \leq m_{F}(x^\ast) + \emph{rad}_{V}(F) + \frac{1}{n}.\\
					\end{split}	
				\end{equation}
				This implies
				\begin{equation}\label{eqn5.3}
					\begin{split}
						M_{F}(x^\ast) - r_{F} - S_{n} \leq v_{n}(x^\ast) \leq m_{F}(x^\ast) + r_{F} + S_{n}.
					\end{split}	
				\end{equation}
				
				We define the following set-valued function. For each $x^\ast \in B_{X^\ast}$,
				\begin{equation}\label{eqn5.4}
					\Phi_{n}(x^\ast) = [v_{n}(x^\ast) -S_{n}\mbox{, }v_{n}(x^\ast) + S_{n}] \cap [M_{F}(x^\ast) - r_{F}, m_{F}(x^\ast) + r_{F}].
				\end{equation}
				Clearly, $\Phi_{n}(x^\ast)$ is closed, convex and bounded for each $x^\ast \in B_{X^\ast}$. Moreover, the inequalities in $(\ref{eqn5.3})$ guarantee that $\Phi_{n}(x^\ast)$ is non-empty for each $x^\ast \in B_{X^\ast}$. The map
				$\Phi_{n}$ can be proved to be a weak$^\ast$-lower semicontinuous function on $B_{X^\ast}$ using the same argument as in the proof of \cite[Theorem~2.2]{SW}.
				
				Now, let $0<\al<1$ and $x^{\ast}_{1},x^{\ast}_{2} \in B_{X^\ast}$. From the facts that $v_n \in A_{\sigma}(B_{X^\ast})$ and the functions $M_{F}$ and $-m_{F}$ are convex on $B_{X^\ast}$, it is easy to verify that \[\al \Phi(x^{\ast}_{1})+(1-\al)\Phi(x^{\ast}_{2}) \ci \Phi(\al x^{\ast}_{1} + (1-\al)x^{\ast}_{2}).\] Moreover, if $x^\ast \in B_{X^\ast}$, then from the facts that $v_n \in A_{\sigma}(B_{X^\ast})$ and $M_{F}(-x^\ast) = -m_{F}(x^\ast)$, it follows that  $-\Phi_{n}(x^\ast) = \Phi_{n}(-x^\ast)$. Therefore, $\Phi_{n}$ is a convex symmetric weak$^\ast$-lower semicontinuous function on $B_{X^\ast}$ and hence, by Theorem~\ref{T2.2.80}, there exists $a_{n} \in A_{\sigma}(B_{X^\ast})$ such that $a_{n}(x^\ast) \in \Phi_{n}(x^\ast)$, for each $x^\ast \in B_{X^\ast}$. Therefore, for each $x^\ast \in B_{X^\ast}$, \[v_{n}(x^\ast) -S_{n} \leq a_{n}(x^\ast) \leq v_{n}(x^\ast) + S_{n}.\] It follows that $\|v_{n} - a_{n}\| \leq S_{n}.$ Moreover, for each $x^\ast \in B_{X^\ast}$, \[M_{F}(x^\ast)-r_{F} \leq a_{n}(x^\ast) \leq m_{F}(x^\ast)+r_{F}.\] Therefore, by Theorem~\ref{T2.3.1}, $a_{n} \in \emph{cent}_{A_{\sigma}(B_{X^\ast})}(F)$. Hence, for each $n=1,2,\ldots$, $R\leq S_{n}$. This proves the formula in $(\ref{eqn5.1})$.
				
				We now prove {\sc Claim~2}. Suppose $R =  \|v_{0} - a_{0}\|,$ for some $v_{0} \in V$ and $a_{0} \in \emph{cent}_{A_{\sigma}(B_{X^\ast})}(F)$. This implies \[r(v_0,F) \leq \sup_{z\in F} \{\|v_{0}-a_{0}\|+ \|a_{0} - z\|\} \leq R + r_{F}= \emph{rad}_{V}(F).\] Therefore, $\emph{rad}_{V}(F) = r(v_0,F)$ and hence, $v_{0} \in \emph{cent}_{V}(F)$. 
				
				Conversely, if $v_{0} \in \emph{cent}_{V}(F)$, then an argument similar to the one above proves that the set-valued map defined as $\Phi(x^\ast) = [v_{0}(x^\ast) - R, v_{0}(x^\ast)+R] \cap [M_{F}(x^\ast) -r_{F}, m_{F}(x^\ast) + r_{F}]$, for each $x^\ast \in B_{X^\ast}$, is convex symmetric weak$^\ast$-lower semicontinuous function on $B_{X^\ast}$. Hence, by Theorem~\ref{T2.2.80}, there exists a selection $a_{0} \in \emph{cent}_{A_{\sigma}(B_{X^\ast})}(F)$ such that $R = \|v_{0} - a_{0}\|$.
			\end{proof}	
			
			The following result provides a geometrical characterization of an $L_1$-predual space.
			\begin{thm}\label{T2.3.5}
				Let $X$ be a Banach space. Then the following statements are equivalent.
				\begin{enumerate}[{(i)}]
					\item $X$ is an $L_{1}$-predual space.
					\item For each $V \in \mc{CV}(X)$ and $F \in \mc{F}(X)$, $\textrm{rad}_{V}(F) = \textrm{rad}_{X}(F) + d(V,\textrm{cent}_{X}(F)).$
					\item For each $V \in \mc{CV}(X)$ and $F \in \mc{F}(X)$, $\textrm{rad}_{V}(F) = \textrm{rad}_{X}(F) + \lim_{\delta \rightarrow 0^{+}} d(V,\textrm{cent}_{X}(F,\delta)).$
				\end{enumerate}
			\end{thm}
			\begin{proof}
				\textit{(i) $\Rightarrow$ \textit{(ii)}} follows from Theorem~\ref{T2.3.4} and  \textit{(ii) $\Rightarrow$ \textit{(iii)}} follows from the following chain of inequalities. For each $F \in \mc{F}(X)$, \[\emph{rad}_{V}(F) \leq \emph{rad}_{X}(F) + \lim_{\delta \rightarrow 0^{+}} d(V,\emph{cent}_{X}(F,\delta)) \leq \emph{rad}_{X}(F) + d(V,\emph{cent}_{X}(F)).\]
				
				In order to prove \textit{(iii) $\Rightarrow$ \textit{(i)}}, by \cite[Theorem~1]{Rao2}, it suffices to show that for each $F \in \mc{F}(X)$, $\emph{rad}_{X}(F) = \frac{1}{2} \emph{diam}(F)$. The proof idea is similar to that in \cite[Theorem~3.4]{EWW}. We include it here for the sake of completeness. 
				
				Let $F =\{x_1,\ldots,x_n\} \ci X.$ Without loss of generality, let $\emph{diam}(F) = \|x_1 - x_2\|$. Define $R = \frac{1}{2} \emph{diam}(F)$. Now, let $F^\prime = \{x_1,x_2\}$ and $V= \{x_3\}$. Then clearly $\emph{rad}_{X}(F^\prime)=R$. By our assumption, \[\emph{rad}_{\{x_{3}\}}(F^\prime) = \emph{rad}_{X}(F^\prime)  + \lim_{\delta \rightarrow 0^{+}} d(\{x_{3}\},\emph{cent}_{X}(F^\prime,\delta)).\] Therefore, $2R \geq R + \lim_{\delta \rightarrow 0^{+}} d(\{x_{3}\},\emph{cent}_{X}(F^\prime,\delta))$ and hence, for each $\varepsilon>0$, there exists $x_{\varepsilon} \in X$ such that $r(x_{\varepsilon}, \{x_1,x_2,x_3\}) \leq R+\varepsilon$. It follows that $\emph{rad}_{X}(\{x_1,x_2,x_3\}) = R$. We next consider $F^\prime = \{x_1,x_2,x_3\}$ and $V=\{x_4\}$ and follow the above arguments to obtain $\emph{rad}_{X}(\{x_1,x_2,x_3,x_4\}) = R$. We keep doing this process until we consider $F^\prime = \{x_1,\ldots,x_{n-1}\}$ and $V=\{x_n\}$ and finally obtain that $\emph{rad}_{X}(F) = R$.
			\end{proof}
			
			\begin{cor}\label{C2.3.6}
				Let $X$ be a Banach space. Then the following statements are equivalent.
				\begin{enumerate}[{(i)}]
					\item For each $V \in \mc{CV}(X)$ and $F \in \mc{K}(X)$, $\textrm{rad}_{V}(F) = \textrm{rad}_{X}(F) + d(V,\textrm{cent}_{X}(F)).$
					\item For each $V \in \mc{CV}(X)$ and $F \in \mc{K}(X)$, $\textrm{rad}_{V}(F) = \textrm{rad}_{X}(F) + \lim_{\delta \rightarrow 0^{+}} d(V,\textrm{cent}_{X}(F,\delta)).$
					\item For each $V \in \mc{CV}(X)$ and $F \in \mc{F}(X)$, $\textrm{rad}_{V}(F) = \textrm{rad}_{X}(F) + d(V,\textrm{cent}_{X}(F)).$
				\end{enumerate}
			\end{cor}
			\begin{proof}
				\textit{(i) $\Rightarrow$ \textit{(ii)}} follows from the following chain of inequalities. For each $F \in \mc{K}(X)$, \[\emph{rad}_{V}(F) \leq \emph{rad}_{X}(F) + \lim_{\delta \rightarrow 0^{+}} d(V,\emph{cent}_{X}(F,\delta)) \leq \emph{rad}_{X}(F) + d(V,\emph{cent}_{X}(F)).\] \textit{(ii) $\Rightarrow$ \textit{(iii)}} follows directly from Theorem~\ref{T2.3.5} and \textit{(iii) $\Rightarrow$ \textit{(i)}} follows from Theorems~\ref{T2.3.5} and \ref{T2.3.4}.
			\end{proof}

         \end{document}